\documentclass[11pt]{article}

\usepackage{amsfonts,amssymb,amsbsy,amsmath,amsthm}
\usepackage[sans]{dsfont}
\usepackage{graphicx}
\usepackage{color}
\usepackage{enumerate}

\numberwithin{equation}{section}

\usepackage[utf8]{inputenc}

\usepackage[english]{babel}

\newtheorem{theorem}{Theorem}[section]
\newtheorem{proposition}[theorem]{Proposition}
\newtheorem{lemma}[theorem]{Lemma}
\newtheorem{follow}[theorem]{Corollary}

\theoremstyle{definition}

\newcommand{\bel}{\begin{equation} \label}
\newcommand{\ee}{\end{equation}}

\newcommand{\one}{\mathds{1}}

\newcommand{\Z}{{\mathbb Z}}
\newcommand{\R}{{\mathbb R}}
\newcommand{\N}{{\mathbb N}}

\newcommand{\re}{{\mathbb R}}

\newcommand{\rd}{{\mathbb R}^{2}}

\begin{document}

\begin{center}{\Large \bf Eigenvalue Asymptotics for a Schr\"odinger Operator with Non-Constant Magnetic Field Along One Direction }

\medskip

{\sc  Pablo Miranda}
\medskip

\end{center}

\bigskip

{\bf Abstract.} {\small We consider the discrete spectrum of
the two-dimensional  Hamiltonian $H=H_0+V$, where $H_0$ is a Schr\"odinger operator with a  non-constant magnetic field $B$
that depends  only on one of the spatial variables, and $V$ is an  electric
potential that  decays at infinity.  We study the accumulation rate of
 the  eigenvalues of H in the gaps of its essential spectrum.
First,  under certain general conditions on $B$ and $V$, we introduce effective Hamiltonians that govern the
main asymptotic term of the eigenvalue counting function. Further, we   use the effective Hamiltonians  to find the asymptotic behavior of the  eigenvalues
in the case where the   potential  V is a power-like decaying function and in the case where it is a compactly
supported function,  showing a semiclassical behavior of the eigenvalues
in the first case and a non-semiclassical behavior in the second one. We also provide a criterion for the
finiteness  of the number of eigenvalues in the gaps of  the essential spectrum of $H$}.\\



\bigskip
\section{Introduction}\setcounter{equation}{0}
\label{s1}
 Let  $\R^2\ni(x,y)\mapsto B(x)\in \R_+$  be a bounded magnetic field  and define  the  Schr\"odinger operator
\bel{dj1}
    H_0:= -\frac{\partial^2}{\partial x^2} +
    \left(-i\frac{\partial}{\partial y} - b(x)\right)^2,
		    \ee
where the second component of the   magnetic vector potential $\R^2\ni(x,y)\mapsto(0,b(x))\in\R^2$ is given by
\bel{b_poten} b(x)=\int_0^x B(t)\,dt.\ee
Let $V:\R^2\to [0,\infty)$   an electric potential that decays at infinity. Set $H=H_0+V$.  It is known that the essential spectrum of $H$, denoted by $\sigma_{\rm ess}(H)$, satisfies
\bel{spec}\sigma_{\rm{ess}}(H)=\bigcup_{j\in \N}{[{\mathcal E}_j^-,{\mathcal E}_j^+]},\ee with ${\mathcal E}_j^\pm\in [0,\infty)$. Suppose that there exists a finite gap in the essential spectrum of $H$, which in our context will be equivalent  to
\bel{cond1} {\mathcal E}_j^+<{\mathcal E}_{j+1}^-,  \ee
for some $ j\geq 1$ (see Section \ref{s2}). Then, it is possible  to define
\bel{njota}\mathcal{N}_j(\lambda):={\rm Tr}\one_{\left({\mathcal E}_j^++\lambda, {\mathcal E}^-_{j+1}\right)}(H), \quad \mbox{for}\,\,0<\lambda <{\mathcal E}_{j+1}^-- {\mathcal E}^+_{j},\ee
where $\one_{\omega}(\cdot)$ is the characteristic function of the set $\omega$. The function ${\mathcal N}_j$ counts the number of eigenvalues of $H$ on the interval $({\mathcal E}_j^++\lambda, {\mathcal E}^-_{j+1})$. Our purpose in this article is to describe the asymptotic behavior of   $\mathcal{N}_j(\lambda)$ as $\lambda$ goes to zero, for some types of non-constant magnetic fields $B$ and electric potentials $V$.

The asymptotic behavior of the function $\mathcal{N}_j$ has been systematically studied in the case of a magnetic field $B$ equal to a \emph{constant} (see  \cite{rai1}, \cite{ivrii}, \cite{raiwar}, \cite{meelroz}, \cite{filpush}, \cite{proz}, \cite{roz}). For this model exists a rather complete understanding of the behavior of  $\mathcal{N}_j$, according to the decaying regime at infinity of the function $V$. This includes  power-like, exponential and compactly supported regimes. An extension of these results was to consider the eigenvalue counting function for  Schr\"odinger operators with \emph{asymptotically constant} magnetic field  and decaying electric potential (see   \cite{ivrii}, \cite{rozta1}, \cite{rozta2}, and for  related problems  see \cite{pushroz2}, \cite{per}).   Other natural extensions  are the Schr\"odinger operators with \emph{unidirectionally constant}  magnetic field presented here. This last model  was first considered by A. Iwatsuka (with $V\equiv 0$) in order to give examples  of  magnetic Schr\"odinger operators with purely absolutely continuous spectrum  \cite{iwa}. The one particle system determined by this Hamiltonian presents some interesting transport and spectral properties which have been studied  in the mathematical  literature (see \cite{manpu}, \cite{exkov}, \cite{shi1}, \cite{shi2}, \cite{dgr}, \cite{lww}, \cite{hissocc2}, \cite{hprs}), as well as in the physics literature  (see e.g. \cite{mu}, \cite{cal}, \cite{gmz}, \cite{pr}, \cite{rmcpv}, \cite{dpccgf}).

For the ``Iwatsuka Hamiltonian'', the problem of  the asymptotic behavior of  a counting function of the form  \eqref{njota}  was already studied   in \cite{dhs}. In that article were considered the  eigenvalues below the bottom of the essential spectrum  of $H$, when the magnetic field $B(x)$ is  a step function that changes sign at zero. This problem was also addressed in \cite{shi1}, for a magnetic field with similar characteristics  to   the one that we will study here (see \eqref{b_bound}). We note that in \cite{dhs} the first band function of $H_0$ (see section \ref{s2}) has  a global  minimum  at a finite point of $\R$, while for the model in  \cite{shi1}, as well as  for the model in this article,  the band functions have  its extremal point   at infinity (for us is relevant  the supremum). This divergence implies that  the analysis of the counting function in our cases is quite different and slightly  more difficult than that  in \cite{dhs}. We also note that the condition required in \cite{shi1} to obtain the mentioned  property on the supremun of the band functions, is that the function $B(x)$ is monotone. In this article we relax somewhat this condition asking only  global monotonicity to $B$ (see \eqref{b_bound}).

In \cite{shi1} the behavior of ${\mathcal N}_j$ was obtained for  potentials  $V$ that decay at infinity as  $(x^2+y^2)^{-m/2}$ (see \eqref{apr14}, \eqref{jul10a} below),  supposing that   $0<m<1$. In Corollary \ref{V_power} we will present a result similar to the semiclassical one given in \cite{shi1},  which completes the description of the first asymptotic term of  ${\mathcal N}_j$ for power-like decaying potentials, that is we consider the case $m >1$. Furthermore, in Theorem \ref{eh} we give an effective Hamiltonian which permit us  to deal with   other types of decaying regimes of $V$. Namely, in Corollary \ref{coro1} we give a sufficient condition that guarantees the finiteness of  the number of eigenvalues of $H$ in  each gap of  $\sigma_{\rm{ess}}(H)$. This is  a geometric condition that depends on the set where $B$ reach its supremum and the support of $V$.

When the condition of Corollary \ref{coro1} does not hold, we can see   that  ${\mathcal N}_j$ is generically unbounded in each gap of  $\sigma_{\rm{ess}}(H)$, as  follows incidentally from  Corollary \ref{coro2} where we  give  asymptotic bounds for ${\mathcal N}_j$ if $V$ is of compact support. Contrary to Corollary \ref{V_power}, the behavior of ${\mathcal N}_j$ is not semiclassical in this situation, since a semiclassical formula would imply the finiteness of the number of eigenvalues. For compact supported potentials $V$, a different  non-semiclassical asymptotic behavior of the  eigenvalue counting function was obtained in \cite{raiwar}, \cite{meelroz}, in the constant magnetic field case. In that context the main asymptotic term is $(\ln|\ln \lambda|)^{-1}|\ln \lambda|$ which goes to infinity  faster than the one  presented here, $|\ln \lambda|^{1/2}$, implying that the accumulation of the eigenvalues is stronger in our case. Similar results to our was previously  obtained  in \cite{bmr1}, \cite{bmr2}, for other   magnetic Hamiltonians with compact supported electric potentials (see \emph{Remark} after Theorem \ref{eh}).

For non-positive potentials $V$   we could define the functions
$${\mathcal N}_0^-(\lambda):={\rm Tr}\one_{(-\infty, {\mathcal E}^-_{1}-\lambda)}(H);\quad \mathcal{N}^-_j(\lambda):={\rm Tr}\one_{\left({\mathcal E}_j^+, {\mathcal E}^-_{j+1}-\lambda\right)}(H),\quad j\in\N.$$
Although we will present our results only for ${\mathcal N}_j$, $j\in\N$,  they are still valid, with  obvious modifications, for ${\mathcal N}_j^-$, $j\in \Z_+:=\{0,1,2,...\}$. We omit these in order to simplify the presentation. Finally, for the case of $V$ without definite sign we can say: If $V$ is  power-like decaying, a refinement of the analysis used  here to obtain the effective Hamiltonian of Theorem \ref{eh}, should lead to the same type of semiclassical results of Corollary \ref{V_power} for this kind of potentials. Otherwise, if $V$   has compact support, the situation  is considerably more delicate and we do not have clear ideas of how to obtain precise results in this context. This is, to some extent, true  even if the   magnetic field is constant, where a description of the eigenvalue counting function has been obtained only for some particular classes of non-sign definite compact electric potentials (see \cite{pushroz2}, \cite{roz}).

\section{Main Results}\label{s2}\setcounter{equation}{0}

\subsection{Effective Hamiltonian}
To introduce the effective Hamiltonians that govern the main  asymptotic term of    $\mathcal{N}_j$, we need to state  more specific conditions on the magnetic field $B$ and then  recall some well-known properties of the unperturbed operator $H_0$.

Throughout this article we will assume the following:
\bel{b_bound}
  \begin{array}{llr}
  \rm{a})& B \in L^{\infty}(\R).\\
  \rm{b})& B_-\leq B(x)\leq B_{+} \,\, \mbox{a.e., for some positive constants} \,\, B_+>B_-.\vspace{.1cm} &\\
  \rm{c})& \lim_{x\to  \infty}B(x)=B_+, \quad \mbox{and} \quad \limsup_{x \to -\infty}B(x)<B_+.\end{array}
\ee

Under  condition \eqref{b_bound} the operator defined by \eqref{dj1} is essentially self-adjoint on $C_0^\infty(\R^2)$  and  its  spectrum, denoted by $\sigma(H_0)$, is purely absolutely con\-ti\-nuous \cite{iwa}, \cite{lww}. Note that the potential $b$ defined by \eqref{b_poten} is an absolutely continuous  strictly increasing function such that
\bel{dic9}
B_-|x|\leq |b(x)|\leq B_+|x|.
\ee

  Let ${\mathcal F}$ be  the partial Fourier transform
$$  (\mathcal{F}u)(x,k)=\frac{1}{(2\pi)^{1/2}}\int_{\R}e^{-iky}u(x,y)\,dy,\quad \mbox{for}\,\,u \,\in\, C_0^\infty(\R^2).$$
Then
\bel{direc} \mathcal{F} H_0\mathcal{F}^*=\int_{\R}^\oplus h(k)\,dk,\ee
where $h(k)$ is a self-adjoint operator acting in $L^2(\R)$, defined by
\bel{sep8}h(k)=-\frac{d^2}{dx^2}+(b(x)-k)^2, \quad k \,\in\, \R.\ee
For any $k \in  \R$ the spectrum of the operator $h(k)$ is discrete and simple. We denote the increasing sequence of eigenvalues by $\{E_j(k)\}_{j=1}^\infty$. For any $j \in \N$  the \emph{band function} $E_j(\cdot)$  is analytic as a function of $k \in \R$ \cite{iwa}, \cite{lww}.

Set ${\mathcal E}_j^-:=\inf_{k\in\R}E_j(k)$; ${\mathcal E}_j^+:=\sup_{k\in\R}E_j(k)$, then
\bel{oct7}\sigma(H_0)=\bigcup_{j=1}\overline{E_j(\R)}=\bigcup_{j=1}[{\mathcal E}_j^-,{\mathcal E}_j^+].\ee
  Condition \eqref{b_bound} b) implies that   $B_-(2j-1)\leq E_j(k) \leq B_+(2j-1)$ for all $k \in \R,$ and \eqref{b_bound} c) implies that $\lim_{k \to \infty} E_j(k)=B_+(2j-1)={\mathcal E}_j^+,$ for all $j \in \N$ (see \cite{iwa}).

Now we need some definitions. Put
\bel{61}
\varphi_j(x) : =  \frac{{\rm H}_{j-1}(x) e^{-x^2/2}}{(\sqrt{\pi}2^{j-1}
(j-1)!)^{1/2}}, \quad x \in \re, \quad j \in {\mathbb N},
    \ee
where
$$
{\rm H}_q(x): = (-1)^q e^{x^2} \frac{d^q}{dx^q} e^{-x^2}, \quad x
\in \re, \quad q \in {\mathbb Z}_+,
$$
are the Hermite polynomials (see e.g. \cite[Chapter I, Eqs. (8.5),
(8.7)]{bershu}). Then the real-valued function $\varphi_j$
satisfies
$$
-\varphi_j''(x) + x^2 \varphi_j(x) = (2j-1) \varphi_j(x), \quad
\|\varphi_j\|_{L^2(\re)} = 1.
$$
For $(x,\xi) \in \R^2$ define the function
    \bel{sof9}
\Psi_{j;x,\xi}(k) = B_+^{-1/4}e^{-ik\xi}\varphi_j(B_+^{1/2}x - B_+^{1/2}b^{-1}(k)), \quad j
\in {\mathbb N}, \quad k \in \re.
    \ee

    The system  $\left\{\Psi_{j;x,\xi}\right\}_{(x,\xi) \in \R^2}$ is overcomplete with respect to the measure $\frac{B_+}{2\pi} dx d\xi$ (see \cite[Subsection 5.2.3]{bershu} for the definition of an overcomplete system with respect to a given measure). Introduce the orthogonal projection
    $$
    {\mathcal P}_{j;x,\xi} : = |\Psi_{j;x,\xi} \rangle \langle \Psi_{j;x,\xi}|, \quad (x,\xi) \in \R^2,
    $$
    acting in $L^2(\re)$, and the pseudo-differential operator ${\mathcal V}_j : L^2(\re) \to L^2(\re)$ defined as the weak integral
    \bel{antiwick}
    {\mathcal V}_j : = \frac{B_+}{2\pi} \int_{\R^2} V(x,\xi) \, {\mathcal P}_{j;x,\xi}\, dx d\xi,
    \ee
    i.e.  ${\mathcal V}_j$ is an operator with  contravariant symbol $V$.

    As already mentioned, for the potential $V$ we will assume the following:
\bel{V_bound}\,\,
  \begin{array}{ll}
 \rm{a})& 0\leq V\in L^\infty(\R^2).\\
 \rm{b})& \lim\displaylimits_{x^2+y^2\to \infty}V(x,y)=0.\vspace{.1cm}
  \end{array}
\ee
The diamagnetic inequality and  Weyl's theorem imply that  $\sigma_{\rm{ess}}(H)=\sigma_{\rm{ess}}(H_0)=\sigma(H_0)$, then \eqref{spec} holds true. Conditions \eqref{V_bound} also imply that ${\mathcal V}_j$ is a non-negative and compact operator.

\begin{theorem}\label{eh}
    Assume that for some $j\in \N$, \eqref{cond1} is true. Assume also that   $B$  satisfies \eqref{b_bound}, and $V$ satisfies \eqref{V_bound}.
   Consider the band function $E_j$ as a multiplication operator in $L^2(\R)$. Then for  each $\delta \in (0,1)$
    \bel{dj5}\begin{array}{c}
   {\rm Tr}\,\one_{({\mathcal E}_j^+ + \lambda,\infty)}(E_j  + (1-\delta){\mathcal V}_j) + O_\delta (1)\\
   \leq {\mathcal N}_j(\lambda) \leq \\{\rm Tr}\,\one_{\small{({\mathcal E}_j^+ + \lambda,\infty)}}(E_j  + (1+\delta){\mathcal V}_j) + O_\delta (1), \quad \lambda \downarrow 0.
    \end{array}\ee
     \end{theorem}
\emph{Remark}: Similar results to Theorem  \ref{eh} appear in \cite{bmr1} and \cite{bmr2}.
 In  \cite{bmr1}    the discrete spectrum of  operators of the form  $H_1=H_{\rm{Hall}}+V$, is described, where
  $$H_{Hall}=H_{\rm{Landau}}+W(x),$$  $H_{\rm{Landau}}$ being the two dimensional Schr\"odinger operator with  constant magnetic field, and $W$ being a monotonic function depending only on the first variable $x$. In the same way, in \cite{bmr2} the operator   $H_2=H_{Half-Plane}+V$ is considered, where   $H_{Half-Plane}$ is the Schr\"odinger operator with constant magnetic field defined for a half-plane, with a Dirichlet boundary condition along the edge. In both articles   an eigenvalue counting function similar to \eqref{njota} is studied. The  effective Hamiltonians obtained in those articles are  particular cases of the one  given by  Theorem \ref{eh}, when $b^{-1}(k)=B_+^{-1}k$ in \eqref{sof9}. All these  three models share the particularity that the unperturbed operators $H_{\rm{Hall}}, H_{\rm{Landau}}$ and $ H_0$  admit a direct integral decomposition with  fibred operators that  converge to  shifted harmonic oscillators as $k\to \infty$. However, despite this similarity, the proof of \ref{eh} requires the use of  some new ideas and presents  technical difficulties that do not appear in  \cite{bmr1} or \cite{bmr2}.

\subsection{Asymptotic behavior of $\mathcal {N}_j(\lambda)$: Finite number of eigenvalues}
In  Corollaries \ref{coro1}, \ref{coro2} we will see that the finiteness  or the infiniteness of the  number  of eigenvalues  of $H$ in the gaps of $\sigma_{\rm {ess}}(H)$,  depend on a relation between the support of $V$ and the  number
\bel{lau24}
{\bf x}^+ : = \inf \{x \in \re \, ; \, B(t) =B_+\,\, \mbox{for almost all} \,\,t\,\, \mbox{in}\,(x,\infty) \}.
\ee
Note that it is possible to have ${\bf x}^+=\infty$.
 \begin{follow} \label{coro1}
 Suppose that  \eqref{cond1} is true, and that $B$ satisfies  \eqref{b_bound}. Assume also  that $V$ satisfies \eqref{V_bound} and   $ \|\int_{\re} V(x,y)\, dy\|_{L^\infty(\R)} < \infty$. Then, if
    \bel{sep9}
    {\bf x}^+ >\sup \left\{ x \in \re \, ; \rm{for\,\, some}\,\, y \in\R, \,\,(x,y) \in {\rm ess}\,{\rm supp}\,V \right\},
      \ee
we have that
     \bel{nov3}
    {\mathcal N}_j(\lambda) = O(1), \quad \lambda \downarrow 0.
    \ee
    \end{follow}
\subsection{Asymptotic behavior of $\mathcal {N}_j(\lambda)$: Infinite number of eigenvalues for $V$ of  compact support}
 Let ${\Omega} \subset \rd$ be a bounded domain. Denote by ${\bf c}_-(\Omega)$ the maximal length of the vertical segments contained in $\overline{\Omega}$. Further, let $B_R((x,y)) \subset {\R^2}$ be a disk of radius $R>0$ centered at $(x,y) \in {\R^2}$. For $a\in \R$, set
    $$
    K(\Omega,a) : = \left\{(\xi, R) \in \re \times \re_+ \, ; \,  \; \mbox{there exists} \; \eta \in \re \; \mbox{such that} \; \Omega \subset B_R((\xi+a,\eta))\right\},
    $$
    and
    $$
    {\bf c}_+(\Omega,a) := \inf_{(\xi, R) \in K(\Omega,a)} R \varkappa\left(\frac{\xi_+}{eR}\right),
    $$
    where $\xi_+ : = \max\{\xi,0\}$, and $ \varkappa(s) : = \left|\left\{ t > 0 \, ; \, t \ln{t} < s\right\}\right|$, for  $ s \in [0, \infty).$ Here  $| \cdot |$ denotes the Lebesgue measure in $\R$.

    Also define
    $$\tilde{\Omega}:=\{(x,y)\in\Omega; x>{\bf x}^+\}.$$

    \begin{follow} \label{coro2} Assume that \eqref{cond1} holds true, and  that $B$ is a function satisfying \eqref{b_bound}.
    Further assume that
    \bel{dj35}
    c_- \one_{\Omega_-}(x,y) \leq V(x,y) \leq c_+ \one_{\Omega_+}(x,y), \quad (x,y) \in {\R^2},
    \ee
    where $\Omega_{\pm} \subset {\R^2}$ are bounded domains with Lipschitz boundaries, and $0 < c_- \leq c_+ < \infty$. Then, if
    $${\bf x}^+<\sup\, \left\{x\in \R;\rm{for\,\, some}\,\, y \in\R, \,\,(x,y)\in{\Omega}_-\right\},$$
   the following asymptotic bounds
    \bel{dj25}
    {\mathcal C}_-  |\ln{\lambda}|^{1/2}(1 + o(1)) \leq {\mathcal N}_j(\lambda) \leq {\mathcal C}_+  |\ln{\lambda}|^{1/2}(1 + o(1)), \quad \lambda \downarrow 0,
    \ee
  hold true  with ${\mathcal C}_- : = (2\pi)^{-1} \sqrt{b} {\bf c}_-(\tilde{\Omega}_-)$ and ${\mathcal C}_+ : = e \sqrt{b} {\bf c}_+(\tilde{\Omega}_+,{\bf x}^+)$.
   \end{follow}

\emph{Remark}:
 The constants ${\mathcal C}_\pm$ already appeared in \cite{bmr1}, \cite{bmr2}, where it is shown that  ${\mathcal C}_-<{\mathcal C_+}$.

\subsection{Asymptotic behavior of $\mathcal {N}_j(\lambda)$: Infinite number of eigenvalues for  power-like decaying $V$}

Now we will consider potentials $V$ whose support is not compact. First we will assume that there exists a positive number  $m$ such that, for any pair $(\alpha, \beta) \in \mathbb{Z}_+^2$,  there exists a positive constant $C_{\alpha,\beta}$, such that
\bel{apr14}
|\partial_x^\beta\partial_\xi^\alpha V(x,\xi)|\leq C_{\alpha, \beta}\langle x, \xi \rangle^{-m-\alpha-\beta} \quad \mbox{ for all} \,\,(x,\xi)\in \R^2,
\ee
where $\langle x, \xi \rangle =(1+x^2+\xi^2)^{1/2}$.

Moreover, let $s\in \R$ and define the volume function
\bel{apr20}
N(\lambda,V,s):=\frac{1}{2\pi}vol\{(x,\xi)\in \R^2; V(x,\xi)> \lambda, x>s\},
\ee
where $vol$ denotes  the Lebesgue measure in $\R^2$.
We will assume that  for some $s_0\in\R$ and positive constants $C$ and  $\lambda_0$
\bel{jul10a}
N(\lambda,V,s_0)\geq C \lambda^{-2/m}, \quad 0<\lambda<\lambda_0.
\ee
We say that a decreasing function $f:\R_+\to \R_+$  satisfies the homogeneity  condition if
\bel{jul10}
\lim_{\epsilon \downarrow 0}\limsup_{\lambda\downarrow 0}\lambda^{2/m}\left(f(\lambda(1-\epsilon))-f(\lambda(1+\epsilon))\right)=0.
\ee
\begin{follow}\label{V_power} Assume that \eqref{cond1} is true. Also suppose that  $B$ is a smooth function with all its derivatives  bounded and for some $M >m$
\bel{may14}
B_+-B(x)= O(\langle x\rangle^{-M}), \quad x \to \infty.
\ee
If $V$ satisfies \eqref{apr14} with $m>1$, and for  $s_0\in \R$, $N(\lambda,V,s_0)$ satisfies  \eqref{jul10a} and  \eqref{jul10},  then we have the following asymptotic formula
\bel{the4}
\mathcal{N}_j(\lambda)=B_+N(\lambda,V,s_0)(1+o(1)), \quad \lambda \downarrow 0.
\ee
\end{follow}

\emph {Remarks:} i) The smoothness condition on $B$ is not essential. For instance, an easy modification of the arguments permits to prove Corollary \ref{V_power} just assuming \eqref{b_bound} and ${\bf x}^+<\infty$.

 ii) Condition \eqref{apr14} implies that if $N(\lambda,V,s_0)$ satisfies \eqref{jul10} for  some $s_0\in\R$, then $N(\lambda,V,s)$  satisfies \eqref{jul10} as well, for any  $s\in\R$. Moreover, if $N(\lambda,V,s_0)$ satisfies \eqref{jul10a} then the asymptotic formula \eqref{the4} is true for any  $s\in\R$, since
  $$\lim_{\lambda \downarrow 0}\frac{N(\lambda,V,s)}{N(\lambda,V,s_0)}=1.$$

iii) Results of the same type of \eqref{the4}  were obtained  in \cite{shi1}, for non necessarily sign-definite potentials $V$, were the number $m$ in \eqref{apr14} is assumed to be $0<m<1$, and the function $B$ monotone.

 iv) As already mentioned in the   \emph{Remark}  after Theorem \ref{eh}, in \cite{bmr1}, \cite{bmr2}  the eigenvalue counting function   for magnetic Schr\"odinger operators  similar to those considered here, was studied. In \cite{bmr1}, \cite{bmr2},  the asymptotic behavior of these counting functions was  described  only for compactly supported potentials $V$, as in Corollary \ref{coro1}, and  a slightly weaker version of Corollary  \ref{coro2} was given. Since the effective Hamiltonians  obtained in \cite{bmr1}, \cite{bmr2} are examples of  the one in Theorem \ref{eh},  the conclusions of Corollary \ref{V_power} are  valid as well for the counting functions of the models considered in the articles \cite{bmr1}, \cite{bmr2}.

\section{Proof of the results} \setcounter{equation}{0}
\label{s3}
\subsection{Proof of Theorem \ref{eh}}
Before we begin the proof, let us  set some notation and auxiliary results that we will use throughout the text.  Let $r>0$ and $T=T^*$ be a linear compact operator acting in a
given Hilbert space\footnote{All Hilbert spaces in this article
are supposed to be separable.}. Set
 $$
n_{\pm}(r; T) : = {\rm Tr}\,\one_{(r,\infty)}(\pm T);
$$
thus the functions $n_{\pm}(\cdot; T)$ are respectively the
counting functions of the positive and negative eigenvalues of the
operator $T$. If $T$ is compact but not necessarily
self-adjoint (in particular, $T$ could act between two different
Hilbert spaces), we will use also the notation
$$
n_*(r; T) : = n_+(r^2; T^* T), \quad r> 0;
$$
thus  $n_{*}(\cdot; T)$ is the counting function of the singular
values of $T$. Evidently,
$$
n_*(r; T) = n_*(r; T^*), \quad n_+(r; T^*T) = n_+(r; T T^*), \quad r>0.
$$
Let us recall also the well-known Weyl inequalities
    \bel{lau11}
    n_+(r_1 + r_2; T_1 + T_2) \leq n_+(r_1; T_1) + n_+(r_2; T_2)
    \ee
    where $r_j > 0$ and $T_j$, $j=1,2$, are
linear self-adjoint compact operators (see e.g. \cite[Theorem 9.2.9]{birsol}), as well as the Ky Fan inequalities
    \bel{lau13}
    n_*(r_1 + r_2; T_1 + T_2) \leq n_*(r_1; T_1) + n_*(r_2;
T_2), \quad r_1, r_2 > 0,
    \ee
    for compact but not necessarily self-adjoint $T_j$, $j=1,2$, (see e.g. \cite[Subsection 11.1.3]{birsol}).
    Further, let $S_p$, $p \in [1,\infty)$, be the Schatten -- von Neumann class of compact operators, equipped with the norm
    $$
    \| T \|_p : = \left( -\int_0^{\infty} r^p \,dn_*(r; T) \right)^{1/p}.
    $$
    Then the  Chebyshev-type estimate
    \bel{dj36}
    n_*(r; T) \leq r^{-p} \|T\|_p^p
    \ee
    holds true for any $r > 0$ and $p \in [1, \infty)$.

\vspace{2pt}

We start the proof by   using the Birman-Schwinger principle, which give us
\bel{B-S_prin}
\mathcal{N}_j(\lambda)=n_-(1; V^{1/2}(H_0-\mathcal{E}^+_j-\lambda)^{-1}V^{1/2})+O(1), \quad \lambda \downarrow 0.
\ee
To analyze the right hand side of \eqref{B-S_prin} it is necessary to obtain further information of the operator $H_0$.

First, note that from  \eqref{direc}
\bel{nov18}(H_0-\mathcal{E}^+_j-\lambda)^{-1}=\mathcal{F}^*\int_\R^\oplus(h(k)-{\mathcal E}_j^+-\lambda )^{-1}\; dk\ \mathcal{F}.\ee
If  $\pi_j(k)$ is the orthogonal projection of $h(k)$ corresponding to the eigenvalue $E_j(k)$,  for $\lambda >0$ and $A\in[-\infty,\infty)$ set
$$ T_j(\lambda,A):=\mathcal{F}^*\int_{(A,\infty)}^\oplus(\mathcal{E}^+_j-E_j(k)+\lambda )^{-1} \pi_j(k)\;dk \,\mathcal{F}.$$

Let $l\in\N$, then the  band function $E_l(\cdot)$ has the following property:\\
Suppose that $l<j$, then for all $k\in\R$
$$
E_l(k)\leq B_+(2l-1)<B_+(2j-1)={\mathcal E}_j^+.
$$
Also,  \eqref{oct7} and  \eqref{cond1} imply that for all $l>j$
$$
E_l(k)\geq {\mathcal E}_l^- \geq {\mathcal E}_{j+1}^- > {\mathcal E}_j^+,
$$
for all $k \in \R$. Then, there exists a positive constant $\kappa$ such that for all $l\neq j$ and $k\in\R$
\bel{jan7}|E_l(k)-{\mathcal E}_j^+|>\kappa.\ee
Inequality \eqref{jan7} implies that, if $I$ is the identity operator in $L^2(\R)$,  the limit
\bel{sep10}
\lim_{\lambda\downarrow 0}(h(k)-{\mathcal E}^+_j-\lambda)^{-1}(I-\pi_j(k))
\ee
exists in the norm operator topology.

For $l=j$ we have the following result.
\begin{lemma}\label{l10}
Let  $j\in\N$, then  $E_j(k)<{\mathcal E}_j^+$ for all $k \in \R$. Moreover,  for any $A \in \R$ there exists $\alpha >0$ such that ${\mathcal E}_j^+-E_j(k)>\alpha$, for all $k<A$.
\end{lemma}
\begin{proof}
First  let us prove that for any $k$ real, ${\mathcal E}_j^+-E_j(k)>0$. Let $B_1$ and $B_2$ be two functions satisfying condition \eqref{b_bound}, and let $b_1$, $b_2$ be the corresponding magnetic potentials as chosen in \eqref{b_poten}. Note that
$$b_s(x)-k=\int_{b_s^{-1}(k)}^{x}B_s(t)\,dt,\quad s=1,2.$$
Then it is easy to see that if $B_1(x)\leq B_2(x)$ a.e. in  $\R$,
 \bel{sep9b}(b_1(x)-k)^2\leq(b_2(x)-b_2(b^{-1}_1(k)))^2,\ee
for all $k$, and all $x$ in $\R$. For $b_1, b_2$, let  $h(k,b_1)$, $h(k,b_2)$ be the operators defined by \eqref{sep8}, and denote by $E_j(k,b_1), E_j(k,b_2)$ their associated $j$-th eigenvalues. The inequality \eqref{sep9b} implies that for all $k \in \R$
\bel{compare} h(k,b_1)\leq h(b_2(b_1^{-1}(k)),b_2),\ee
and from the  min-max principle we obtain  that for all $k \in \R$, and all $j\in \N$
\bel{vital}
E_{j}(k,b_1)\leq E_{j}(b_2(b_1^{-1}(k)),b_2).\ee

Now,  since $\limsup_{x\to -\infty}B(x)<B_+$, there exists a real number $\beta$ and a   non-decreasing smooth function $B_\beta$ such that
$${B}_\beta(x)=\left\{
  \begin{array}{cl}
    \limsup_{x \to -\infty}B(x) & \mbox{if}\, x\leq \beta \\
    B_+ & \mbox{if}\, x\geq \beta+1,\\
  \end{array}
\right.
$$
and  $ B(x)\leq{B}_\beta(x)$ a.e. in $\R$.  From the proof of \cite[Theorem 3.2]{manpu} we know that $B_\beta$  non-decreasing implies that $E_j(k,{b}_\beta)$ is a  non-decreasing function as well. Since $E_j(\cdot,{b}_\beta)$  is also analytic,  \eqref{oct7} implies that  $E_j(k,{b}_\beta)< {\mathcal E}_j^+$ for all $k\in\R$. Using  \eqref{vital} we obtain that $E_j(k)<{\mathcal E}_j^+$.

To prove the second assertion of the Lemma, just note that $E_j(\cdot,{b}_\beta)$ satisfies  the required condition and use \eqref{vital} again.
\end{proof}

Using  the Weyl inequalities \eqref{lau11} together with \eqref{nov18} and  \eqref{sep10}, and together with  Lemma \ref{l10}, it can be  easily seen  that for any $r\in (0,1)$
\bel{1}\begin{array}{cl}
  n_+(1+r;V^{1/2}T_j(\lambda,A)V^{1/2})+ O(1)   & \leq n_-(1; V^{1/2}(H_0-\mathcal{E}^+_j-\lambda)^{-1}V^{1/2}) \\
   & \leq  n_+(1-r;V^{1/2}T_j(\lambda,A)V^{1/2})+O(1),
\end{array}
\ee
as $\lambda \downarrow 0.$

Next, let $h_\infty(k)$ be the shifted harmonic oscillator $$h_\infty(k):=-\frac{d^2}{dx^2}+(B_+x-B_+b^{-1}(k))^2,$$
self-adjoint in $L^2(\R)$, for  $k\in\R$. The spectrum of $h_\infty(k)$ coincide with the set of Landau levels $\{B_+(2j-1)={\mathcal E}_j^+\}_{j=1}^\infty$. Let $\pi_{j,\infty}(k)$ be the orthogonal projection of $h_\infty(k)$ corresponding to the eigenvalue ${\mathcal E}^+_j$, which can be described explicitly by
 \bel{dic11}    \pi_{j,\infty}(k)  = |\Psi_{j,\infty}(\cdot,k) \rangle \langle \Psi_{j,\infty}(\cdot,k)|,
  \ee
 where  $\Psi_{j,\infty}(x,k)= B_+^{1/4}\varphi_j(B_+^{1/2}x - B_+^{1/2}b^{-1}(k))$ ($\varphi_j$  defined in \eqref{61}).

For $\lambda >0$ and  $A\in [-\infty,\infty)$, set
\bel{may10}T_{j,\infty}(\lambda,A):=\mathcal{F}^*\int_{(A,\infty)}^\oplus(\mathcal{E}^+_j-E_j(k)+\lambda )^{-1} \pi_{j,\infty}(k)\,dk \mathcal{F}.\ee

Our next goal is to replace $T_j(\lambda,A)$ by $T_{j,\infty}(\lambda,A)$ in  inequality \eqref{1}.

\begin{theorem}\label{pr-vs-band} For any $j \in \N$
$$\lim_{k\to \infty}\frac{||\pi_j(k)-\pi_{j,\infty}(k)||}{(\mathcal{E}^+_j-E_j(k))^{1/2}}=0.$$
\end{theorem}
The proof of this Theorem follows from the next two lemmas.
\begin{lemma}\label{lambda_lemma} Define $\Lambda_k:=h(k)^{-1}-h_\infty(k)^{-1}$. Then $\Lambda_k\geq0$ and
\bel{lambda}
\lim_{k\to \infty} ||\Lambda_k||=0.
\ee
\end{lemma}
\begin{proof} To see that $\Lambda_k\geq 0$, use  \eqref{b_bound} b) and \eqref{compare}.
To prove \eqref{lambda} we  introduce the unitary operators $U_k:L^2(\R)\to L^2(\R)$ defined for any $k\in\R$ by
$$(U_kf)(x)=f(x+b^{-1}(k)),$$
and set
$$\tilde{h}(k):=U_kh(k)U_k^*=-\frac{d^2}{dx^2}+(b(x+b^{-1}(k))-k)^2$$
and
$$\tilde{h}_\infty:=U_kh_\infty(k)U_k^*=-\frac{d^2}{dx^2}+(B_+x)^2.$$
Instead of \eqref{lambda}  we will prove the equivalent statement $\lim_{k\to \infty} ||\tilde{h}(k)^{-1}-\tilde{h}_\infty^{-1}||=0$.

Put $d_k(x):=\tilde{h}_\infty-\tilde{h}(k)=(B_+x)^2-(b(x+b^{-1}(k))-k)^2$. Using two times the resolvent identity we get
\bel{res-id}
\tilde{h}(k)^{-1}-\tilde{h}_\infty^{-1}=\tilde{h}_\infty^{-1}d_k\tilde{h}_\infty^{-1}+\tilde{h}(k)^{-1}d_k\tilde{h}_\infty^{-1}d_k\tilde{h}_\infty^{-1}.
\ee
We need to prove first that  $\tilde{h}_\infty^{-1}d_k\tilde{h}_\infty^{-1}$ converges to zero in norm  as $k\to\infty$.

Note that
\begingroup
\renewcommand*{\arraystretch}{2}
\bel{sep11}\begin{array}{cl}
  |d_k(x)| &= {\displaystyle \left|B_+x+(b(x+b^{-1}(k))-k)\right|\left|B_+x-(b(x+b^{-1}(k))-k)\right|} \\
   & ={\displaystyle \left|\int_{b^{-1}(k)}^{b^{-1}(k)+x}B_++B(t)\,\,dt\right|\left|\int_{b^{-1}(k)}^{b^{-1}(k)+x}B_+-B(t)\,\,dt\right|} \\
   & {\displaystyle \leq 2B_+|x|\left|\int_{b^{-1}(k)}^{b^{-1}(k)+x}B_+-B(t)\,dt\right|.}
\end{array}
\ee
\endgroup
Then, since   $\lim_{x\to \infty} B(x)=B_+$, and from \eqref{dic9} $b^{-1}(k)\to \infty$ for $k\to \infty$, the function  $|d_k(x)|$ converges pointwise to zero when $k\to \infty$.

Denote by  $D(\tilde{h}(k))$, $D(\tilde{h}_\infty)$  the domains of $\tilde{h}(k)$ and $\tilde{h}_\infty$, respectively.  Using \eqref{sep9b}, $B_-\leq B(x)\leq B_+$ implies that  \bel{sep11a}
 \left(B_-x\right)^2\leq \left(b(x+b^{-1}(k))-k\right)^2\leq\left(B_+x\right)^2, \quad \mbox{for all}\,\,x\in\R. \ee
 Then the domains are equal and  coincide with the domain of the harmonic oscillator, i.e., $D(\tilde{h}(k))=D(\tilde{h}_\infty)=D(-d^2/dx^2)\cap D(x^2)$ \cite[Theorem 1]{evgi}.

Let  $f\in L^2(\R)$. Since   $\tilde{h}_\infty^{-1}f\in D(x^2)$,    for any $\epsilon >0$ one can find $N>0$ (independent of $k$) such that
\bel{nov3b}\int\displaylimits_{|x|>N}|(b(x+b^{-1}(k))-k)^2(\tilde{h}^{-1}_\infty f)(x)|^2\,dx\leq \int\displaylimits_{|x|>N}|(B_+x)^2(\tilde{h}^{-1}_\infty f)(x)|^2\,dx<\epsilon.\ee
Further,  $\tilde{h}^{-1}_\infty f$ is also  continuous, then
\bel{nov3c}\int_{-N}^N\left|d_k(x)(\tilde{h}^{-1}_\infty f)(x)\right|^2\,dx\leq \sup_{x\in [-N,N]}\left|(\tilde{h}_\infty^{-1} f)(x)\right|^2\int_{-N}^N\left|d_k(x)\right|^2\,dx.\ee
Using \eqref{nov3b}, \eqref{nov3c} and \eqref{sep11} we can conclude that $d_k\tilde{h}^{-1}_\infty$ converges strongly  to zero as $k\to \infty$. Consequently, the family $d_k\tilde{h}^{-1}_\infty$ is uniformly bounded with respect to $k$, and since $\tilde{h}^{-1}_\infty$ is compact we  get   $\|\tilde{h}^{-1}_\infty d_k\tilde{h}^{-1}_\infty\|\to 0$ for $k\to\infty$.

To finish the proof of the Lemma it only remains to show that for all $G \in D(\tilde{h}(k))=D(\tilde{h}_\infty)$, $\|\tilde{h}(k)^{-1}d_k G \|_{L^2(\R)}\leq C \|G\|_{L^2(\R)}$, for some constant  $C$ independent of $k$ and $G$.

From \cite[Theorem 1]{evgi} we know that for any $g \in D(\tilde{h}(k))$
\bel{eve} \frac{1}{2}||(b(x+b^{-1}(k))-k)^2g||_{L^2(\R)}^2\leq ||\tilde{h}(k)g||_{L^2(\R)}^2.\ee
Then for $f$  in $L^2(\R)$,  if $g=\tilde{h}(k)^{-1}f$ in \eqref{eve}
$$\frac{1}{2}||(b(x+b^{-1}(k))-k)^2\tilde{h}(k)^{-1}f||_{L^2(\R)}^2\leq ||f||_{L^2(\R)}^2.$$
Besides, using \eqref{sep11a} we get $||(B_+x)^2\tilde{h}(k)^{-1}f||_{L^2(\R)}^2\leq 2 B_+^2/B_-^2||f||_{L^2(\R)}^2,$  which implies the existence of a  uniform bound for $d_k\tilde{h}(k)^{-1}$, from where we can easily get  the needed result for $\tilde{h}(k)^{-1}d_k$.
\end{proof}
\begin{lemma}\label{l2}
Let $\Lambda_k$ be defined as in Lemma \ref{lambda_lemma}. For all $j\in\N$:
\begin{enumerate}
\item There exist a constant $C_j$ such that for all  $k$ big enough
$$||\pi_j(k)-\pi_{j,\infty}(k)||\leq C_j ||\Lambda_k \pi_{j,\infty}(k)||.$$

\item It is satisfied the  asymptotic formula \bel{asymb_band} \mathcal{E}^+_j-E_j(k)={\mathcal{E}^+_j}^2||\pi_{j,\infty}(k)\Lambda_k\pi_{j,\infty}(k)|| (1+o(1)), \quad k \to \infty.
\ee
\end{enumerate}
\end{lemma}
\begin{proof}
 The   proof of this Lemma uses Lemma \ref{lambda_lemma}   repeating almost word by word the proof of Propositions 3.6 and 3.7 in \cite{bmr2}.
\end{proof}
Putting together Lemmas \ref{l10}, \ref{lambda_lemma} and \ref{l2} we can proof  Theorem \ref{pr-vs-band} just by  noticing that
 $$
\frac{||{\pi}_j(k)-{\pi}_{j,\infty}(k)||}{(E_j(k)-\mathcal{E}^+_j)^{1/2}} \leq
 \frac{C_j}{{\mathcal E}^+_j} \frac{\|\Lambda_k{\pi}_{j,\infty}(k) \|}{\|\Lambda_k^{1/2}{\pi}_{j,\infty}(k) \|} (1 + o(1)) \leq \frac{C_j}{{\mathcal E}_j^+} \|\Lambda_k\|^{1/2} (1 + o(1)), \quad k \to \infty.
    $$
\begin{proposition}
For all $A\in[-\infty,\infty)$, $r\in\R$, $\delta\in(0,1)$ and $j\in\N$
\bel{2a}\begin{array}{ll}
   & n_+(r(1+\delta);V^{1/2}T_{j,\infty}(\lambda,A)V^{1/2})+O(1) \\
  \leq & n_+(r;V^{1/2}T_{j}(\lambda,A)V^{1/2}) \\
  \leq & n_+(r(1-\delta);V^{1/2}T_{j,\infty}(\lambda,A)V^{1/2})+O(1), \quad \lambda\downarrow 0.
\end{array}
\ee
\end{proposition}
\begin{proof} First note that $n_+(r;V^{1/2}T_j(\lambda,A)V^{1/2})=n_*(r^{1/2};V^{1/2}T_j(\lambda,A)^{1/2})$.
By Lemma \ref{l10}, for any $\tilde{A}\in (A,\infty)$
$$\begin{array}{l}n_*(r;V^{1/2}(T_j(\lambda,A)^{1/2}-T_j(\lambda,\tilde{A})^{1/2}))=O(1)\\
n_*(r;V^{1/2}(T_{j,\infty}(\lambda,A)^{1/2}-T_{j,\infty}(\lambda,\tilde{A})^{1/2}))=O(1), \quad \lambda \downarrow 0,
\end{array}$$
 since both $T_j(\lambda,A)^{1/2}-T_j(\lambda,\tilde{A})^{1/2}$ and $T_{j,\infty}(\lambda,A)^{1/2}-T_{j,\infty}(\lambda,\tilde{A})^{1/2}$ have a limit in the norm sense when $\lambda \downarrow 0$. Thanks to Theorem \ref{pr-vs-band} it is possible to choose $\tilde{A}$ big enough such that
$$n_*(r;V^{1/2}(T_j(\lambda,\tilde{A})^{1/2}-T_{j,\infty}(\lambda,\tilde{A})^{1/2}))=0.$$

Using the  Ky-Fan inequalities \eqref{lau13} we  get  \eqref{2a} (see Proposition 3.1 and Theorem 3.2 in \cite{bmr1} for a detailed  proof of a similar result).
\end{proof}

Now we are ready to finish the proof of Theorem \ref{eh}. Putting together \eqref{B-S_prin}, \eqref{1} and \eqref{2a}, we obtain that for any $A\in[-\infty,\infty)$ and $\delta \in (0,1)$
\bel{2}\begin{array}{lcl}
    &  &n_+((1+\delta);V^{1/2}T_{j,\infty}(\lambda,A)V^{1/2})+O(1)   \\
     \leq&{\mathcal N}_j(\lambda)  \leq &n_+((1-\delta);V^{1/2}T_{j,\infty}(\lambda,A)V^{1/2})+O(1),  \quad \lambda \downarrow 0.
  \end{array}
\ee
For $A\in[-\infty,\infty)$ define
$$P_{j,\infty}(A):=\mathcal{F}^*\int_{(A,\infty)}^\oplus \pi_{j,\infty}(k)\;dk\, \mathcal{F},$$
then,  setting  $A=-\infty$ we obtain that for any $r>0$
\bel{may4}
\begin{array}{ll}
  &n_+(r;V^{1/2}T_{j,\infty}(\lambda,-\infty)V^{1/2})\\
  =&n_+(r;T_{j,\infty}(\lambda,-\infty)^{1/2}VT_{j,\infty}(\lambda,-\infty)^{1/2})
 \\
  = & n_+(r;({\mathcal E}^+_j-E_j(\cdot)+\lambda)^{-1/2}{\mathcal F}P_{j,\infty}(-\infty)VP_{j,\infty}(-\infty){\mathcal F}^*({\mathcal E}^+_j-E_j(\cdot)+\lambda)^{-1/2}).

\end{array}
\ee
Let ${\mathcal U}:L^2(\R)\to {\mathcal F}P_{j,\infty}(-\infty){\mathcal F}^*\left(L^2(\R^2)\right)$, defined by $({\mathcal U}g)(x,k)=B_+^{1/4}g(k)\varphi_j(B_+^{1/2}x-B_+^{1/2}b^{-1}(k))$. The operator ${\mathcal U}$ is unitary   and \bel{nov13}{\mathcal U}^*{\mathcal F}P_{j,\infty}(-\infty)VP_{j,\infty}(-\infty){\mathcal F}^*{\mathcal U}={\mathcal V}_j,\ee
$${\mathcal U}^*({\mathcal E}_j^+-E_j(\cdot)+\lambda)^{-1}{\mathcal U}=({\mathcal E}_j^+-E_j(\cdot)+\lambda)^{-1}.$$ Use \eqref{2}, \eqref{may4} and \eqref{nov13} together with the Birman-Schwinger principle to get \eqref{dj5}.

\subsection{Proof of Corollary \ref{coro1}} 

From  inequality \eqref{2}, we see that to prove this corollary  it is enough to show  that for some $A\in[\infty,\infty)$ and $r\in(0,1)$
\bel{22}
 n_+(r^2;V^{1/2}T_{j,\infty}(\lambda,A)V^{1/2})=n_*({r};T_{j,\infty}(\lambda,A)^{1/2}V^{1/2})=O(1), \quad \lambda \downarrow 0.
\ee
The Chebyshev-type estimate  \eqref{dj36}, with $p=2$, states that
\bel{23}
\begin{array}{ll}&\displaystyle{n_*(r;T_{j,\infty}(\lambda,A)^{1/2}V^{1/2})\leq r^{-2}\|T_{j,\infty}(\lambda,A)^{1/2}V^{1/2}\|_2^2}\\
=&\displaystyle{\frac{1}{2\pi r^2}\int_A^{\infty} \int_{\rd} ({\mathcal E}^+_j- E_j(k) + \lambda)^{-1} \psi_{j,\infty}(x,k)^2 V(x,y)\; dx\, dy\, dk,}\end{array}\ee
where we have used \eqref{may10} and \eqref{dic11}. Here and in the sequel we will assume without loss of generality that ${\bf x}^+=0$.  Indeed, for ${\bf x}^+$ finite, this follows from  a translation along the x-axis and using the gauge invariance of $H$. If ${\bf x}^+$ is infinite, thanks to \eqref{sep9b}-\eqref{compare} we may replace $B$ by  a function $\widetilde{B}$ such that $\widetilde{B}(x)\geq B(x)$ and   that the number ${\bf x}^+_{\widetilde{B}}:=\inf \{x \in \re \, ; \, \tilde{B}(t) =B_+\,  \,\mbox{for almost all} \,\, t\,\,\mbox{in}\,(x,\infty) \}$ is equal to zero, and then use \eqref{vital} in \eqref{ine_band} below.

Put $X^+:=\sup\{ x \in \re \, ; \,\,\mbox{for some} \,\, y\in \R,  (x,y) \in {\rm ess}\,{\rm supp}\,V\}$. Take $\tilde{x}$ such that  ${X}^+<\tilde{x} <0={\bf x}^+$, and define the  step function
\bel{aug25a}{W}(x):=\left\{\begin{matrix} b(\tilde{x})^2-(B_+\tilde{x})^2 &\mbox{for}\, x< \tilde{x} \\
0 & \mbox{for}\, x \geq \tilde{x}.\end{matrix} \right.\ee
Setting   $h^W(k)$ as  the  operator given by
$$-\frac{d^2}{dx^2}+(B_+x-B_+b^{-1}(k))^2+W(x),$$
self-adjoint  in $L^2(\R)$, it is not difficult to  see that for $k>0 $
\bel{ine2}
 h(k) \leq h^{W}(k).
\ee
The spectrum of $h^W(k)$ is discrete and simple. Denote by  $\{E^W_j(k)\}_{j=1}^\infty$ the increasing sequence of eigenvalues of $h^W(k)$. Inequality   \eqref{ine2} implies that
\bel{ine_band}
E_j(k)\leq E^W_j(k),
\ee
and then $({\mathcal E}^+_j-E_j(k)+\lambda)^{-1}\leq({\mathcal E}^+_j-E^W_j(k)+\lambda)^{-1}$ for all $j\in \N$, $k>0$ and $\lambda>0$.

By Proposition 4.2 of \cite{bmr1}, we know that there exists a positive constant $C_j$ such that for all $k$ big enough
$${\mathcal E}^+_j - E^{W}_j(k)\geq C_j (B_+b^{-1}(k))^{2j-3}e^{-B_+(b^{-1}(k)-\tilde{x})^2}.$$
Then for $A>0$ large, for any $\lambda>0$
\begingroup
\renewcommand*{\arraystretch}{2}
\begin{equation*}\begin{array}{ll}
&\displaystyle{\int_A^{\infty} \int_{\rd} ({\mathcal E}^+_j- E_j(k) + \lambda)^{-1} \psi_{j,\infty}(x,k)^2 V(x,y) \,dx\, dy\, dk}\\
\leq&\displaystyle{\frac{1}{B_+^{2j-3}C_j}\int_A^{\infty}\int_{\R^2} k^{3-2j}e^{2k(x-\tilde{x})}e^{B_+(\tilde{x}^2-x^2)}{\rm H}_j(B^{1/2}_+x-k/B_+^{1/2})^2 V(x,y)\, dy\, dx\, dk,}
\end{array}
\end{equation*}
\endgroup
where we have used that $b^{-1}(k)=k/B_+$ for $k>0$, due to ${\bf x}^+=0$.
The last integral can be decomposed into a finite sum  of terms of the form
 \begin{equation*}\begin{array}{ll}&\displaystyle{C_{l,n}e^{B_+\tilde{x}^2}\int_A^{\infty}\int_{\R^2} k^{l}x^n e^{2k(x-\tilde{x})}e^{-B_+x^2} V(x,y)\, dy \,dx\, dk}\\
    \leq&\displaystyle{\left\|\int_\R V(x,y)dy\right\|_{L^\infty(\R)}|C_{l,n}|e^{B_+ \tilde{x}^2}
    \int_A^{\infty}k^{l} e^{2k(X^+-\tilde{x})}\,dk\int_{-\infty}^{X^+} |x|^n e^{-B_+x^2} \,dx},
    \end{array}\end{equation*}
for some  constants $C_{l,n}$, and integers  $l,n$. Each one of this terms is finite because of  our choice of $\tilde{x}$.

\subsection{Proof of Corollary \ref{coro2}}
Let us first show  how to obtain the  upper bound in \eqref{dj25}. As in the proof of Corollary \ref{coro1},  take the function $W$ defined in \eqref{aug25a}, and for $A\in [-\infty,\infty)$, $\lambda >0$ set
$$T^W_{j,\infty}(\lambda,A):=\mathcal{F}^*\int_{(A,\infty)}^\oplus(\mathcal{E}^+_j-E^W_j(k)+\lambda )^{-1} \pi_{j,\infty}(k)\;dk \,\mathcal{F}.$$
From \eqref{ine_band}, $T_{j,\infty}(\lambda,A)\leq T^{W}_{j,\infty}(\lambda,A)$, thus \eqref{2} implies that for all $A \in [-\infty,\infty)$ and $r\in(0,1)$
\bel{aug25}
{\mathcal N}_j(\lambda)\leq n_+(1-r;V^{1/2}T^W_{j,\infty}(\lambda,A)V^{1/2})+O(1), \quad \lambda \downarrow 0.
\ee

The asymptotic behavior of the function $n_+(1-r;V^{1/2}T^W_{j,\infty}(\lambda,A)V^{1/2})$ was studied in \cite{bmr1} where it is shown that (Theorems 5.1 and 6.1)
\bel{oct7d}\limsup_{\lambda \downarrow 0}\frac{ n_+(1-r;V^{1/2}T^W_{j,\infty}(\lambda,A)V^{1/2})}{|\ln \lambda|^{1/2}}\leq {\mathcal C}_+.\ee
Putting together \eqref{aug25} and \eqref{oct7d} we get the upper bound in \eqref{dj25}.

For the lower bound consider the operators  $h_+^N(k):=-d^2/dx^2+(B_+x-k)^2$ and $h_-^N(k):=-d^2/dx^2+(B_-x-k)^2$ defined in $L^2(\R_+)$ and  $L^2(\R_-)$,
respectively, both  with a Neumann boundary condition at zero.
From the monotonicity property with respect to the   Neumann  conditions,  and from \eqref{sep9b} we   obtain that
\bel{dic12}h(k)\geq h_-^N(k)\oplus h_+^N(k)\ee
(recall that ${\bf x}^+=0$, which implies that $b(x)=B_+x$\, for $x\geq0$). The operators $h_\pm^N(k)$ have  discrete and simple spectrum for any $k\in \R$. Denoting by $\{E_j^{N\pm}(k)\}_{j=1}^\infty$  their increasing
sequences  of eigenvalues, and using  that
$$\lim_{k\to\infty}E_1^{N-}(k)=\infty,\quad \lim_{k\to\infty}E_j^{N+}(k)={\mathcal E}_j^+, $$
(see e.g. \cite{fh}), we can conclude from  \eqref{dic12}  that  for any $j\in\N$ there exists a constant $K_j$ such that
\bel{aug25c}E_j(k)\geq E^{N+}_j(k),\quad \mbox{for}\,\,k\geq K_j.\ee

Set
$$T^N_{j,\infty}(\lambda,A):=\mathcal{F}^*\int_{(A,\infty)}^\oplus(\mathcal{E}^+_j-E^{N+}_j(k)+\lambda )^{-1} \pi_{j,\infty}(k)\;dk\, \mathcal{F}.$$
Inequality  \eqref{2} along with  \eqref{aug25c} imply that for any $r\in(0,1)$ and $A\geq K_j$
\bel{aug25b}
{\mathcal N}_j(\lambda)\geq n_+(1+r;V^{1/2}T^N_{j,\infty}(\lambda,A)V^{1/2})+O(1), \quad \lambda \downarrow 0.
\ee
Besides, it is shown in  \cite{popoff}   that for some positive  constant $C_j$
$$\mathcal{E}^+_j-E^{N+}_j(k)=C_j k^{2j-1}e^{-k^2/B_+}(1+o(1)), \quad k \to \infty.$$
Then,  we can repeat  the proofs  of Proposition 3.7 and Corollary 3.9  in \cite{bmr2} in order to obtain
\bel{nov18a}\liminf_{\lambda \downarrow 0}\frac{ n_+(1+r;V^{1/2}T^N_{j,\infty}(\lambda,A)V^{1/2})}{|\ln(\lambda)|^{1/2}}\geq {\mathcal C}_-.\ee
The inequalities   \eqref{aug25b}, \eqref{nov18a} imply the lower bound in \eqref{dj25}.

\subsection{Proof of Corollary \ref{V_power}: Upper bound}\label{sub_upp}
The starting point of this proof is, as for Corollaries \ref{coro1}, \ref{coro2}, the inequalities \ref{2}.
We will  denote  the operator $T_{j,\infty}(\lambda,-\infty)$ simply by $T_{j,\infty}(\lambda)$, and $P_{j,\infty}(-\infty)$ by $P_{j,\infty}$.
Also from  now on, without any lost of generality, we will take $s=0$ for the function \eqref{apr20}.
That means,  we will prove \eqref{the4} for $N(\lambda,V,0)$ (see \emph {Remark} ii after Corollary \ref{V_power}).

Let $\varepsilon>0$ and take    a smooth function $\chi_\varepsilon$ with bounded derivatives  such that $0\leq \chi_\varepsilon(x) \leq 1$ for all $x \in \R$, $\chi_\varepsilon(x)=0$ for $x\leq-2\varepsilon$ and $\chi_\varepsilon(x)=1$ for $x\geq-\varepsilon$. Define
\bel{may9}
V_\varepsilon(x,\xi):=\chi_\varepsilon(x) V(x,\xi).
\ee

The   Weyl's inequalities say that  for any $r>0$, $\delta \in (0,1)$, and $\lambda >0$
\bel{may3}
\begin{array}{lcl}n_+(r;T_{j,\infty}(\lambda)^{1/2}VT_{j,\infty}(\lambda)^{1/2}) &\leq &n_+(r(1-\delta);T_{j,\infty}(\lambda)^{1/2}V_\varepsilon T_{j,\infty}(\lambda)^{1/2})\\
 &+&n_+(r\delta;T_{j,\infty}(\lambda)^{1/2}(V-V_\varepsilon)T_{j,\infty}(\lambda)^{1/2}).\end{array}\ee
The function  $V-V_\varepsilon$ is equal to zero for $x\geq-\varepsilon$. Arguing as in the proof of Corollary \ref{coro1}, we can see that for any $r>0$
\bel{may6} n_+(r;T_{j,\infty}(\lambda)^{1/2}(V-V_\varepsilon)T_{j,\infty}(\lambda)^{1/2})
=n_*(\sqrt{r};T_{j,\infty}(\lambda)^{1/2}(V-V_\varepsilon)^{1/2})=O(1), \quad \lambda \downarrow 0.\ee
Now, since $E_j(k)\leq {\mathcal E}_j^+$, $T_{j,\infty}(\lambda)\leq \lambda^{-1}P_{j,\infty}$, thus  the min-max principle implies that for all $r >0$ and $\lambda >0$
\bel{may7}
\begin{array}{llll}
 &n_+(r;T_{j,\infty}(\lambda)^{1/2}V_\varepsilon T_{j,\infty}(\lambda)^{1/2})&\leq& n_+(r\lambda;P_{j,\infty}V_\varepsilon P_{j,\infty}).\\
 \end{array}\ee

Next, let us introduce a class of symbols suitable for our purposes. For $(x,\xi) \in \R^2$ consider the quadratic form in $\R^2$
$$ g_{x,\xi}(y,\eta)=|y|^2+\frac{|\eta|^2}{\langle x,\xi\rangle^2},$$
and for $p$, $q$ $\in\R$, define the weight $w:=\langle x\rangle ^{p}\langle x,\xi\rangle ^q.$
Then, according to \cite[Definition 18.4.6]{hor3}, consider the  class of symbols $S_p^q:=S(w,g)$. A symbol $a$ is in $S_p^q$ if
for any $(\alpha,\beta) \in \Z_+^2$, the quantity
\bel{sep25b}n^{p,q}_{\alpha,\beta}(a):=\sup_{(x,\xi) \in \R^2}|\langle x\rangle^{-p}\langle x,\xi\rangle ^{-q+\alpha}\partial_\xi^\alpha \partial_x^\beta a(x,\xi)| \ee
is finite.

For $a\in S_p^q$ we define the operator  $Op^W(a)$ according to the Weyl quantization
$$(Op^{W}(a)u)(x):=\frac{1}{2\pi}\int_{\R^2}a\left(\frac{x+y}{2},\xi\right)e^{-i(x-y)\xi}u(y)\,dy\,d\xi,$$
for $u$ in the Schwartz space ${\mathcal S}(\R)$.

Since $V$ satisfies \eqref{apr14} it is obvious that $V_\varepsilon$ is in $S^{-m}_0$. Moreover, using \eqref{b_bound} b),
it is also true that the function $$\widetilde{V}_\varepsilon(x,\xi):=V_\varepsilon(b^{-1}(x),-\xi)$$
belongs to $ S^{-m}_0$. Due to  $m>0$, the operator $Op^{W}(\widetilde{V}_{\varepsilon})$ is compact in $L^2(\R)$.

Using the same notation of Theorem \ref{eh}, write ${\mathcal V}_{\varepsilon,j}$ for the pseudodifferential operator with contravariant symbol $V_{\varepsilon}$
defined  by  \eqref{antiwick}.

\begin{lemma}\label{shirai}
For any $\varepsilon >0$ and $j\in\N$
\bel{sep25}
\mathcal{V}_{\varepsilon,j}-Op^{W}(\widetilde{V}_{\varepsilon})=Op^{W}(R_1)+Op^{W}(R_2),
\ee
where the symbol $R_1\in S^{-m-1}_0$ and $R_2\in S^{-m}_{-m}.$
\end{lemma}
\begin{proof}
We give a sketch of the proof which is based on  the proof of \cite[Lemma 5.1]{shi2}. Suppose that $V$ is in the Schwartz space $\mathcal {S}(\R^2)$. Then, from \eqref{antiwick}  the  Weyl symbol $p_V$ of $\mathcal{V}_{\varepsilon,j}$  is given by
$$p_V(\eta,\eta^*)=\frac{B_+}{2\pi}\int_{\R^3} e^{-iw\eta^*} \Psi_{j;x,\xi}\left(\eta+w/2\right)\overline{\Psi_{j;x,\xi}\left(\eta-w/2\right)} V_\varepsilon(x,\xi)\;dx\,d\xi\,dw,$$
$\Psi_{j;x,\xi}$ being defined in \eqref{sof9}. We use a  first order Taylor expansion  of $V_\varepsilon$,  noticing  that $\partial_1V_\varepsilon=(\partial_1V)\chi_\varepsilon+V(\partial_1\chi_\varepsilon)$, $\partial_2V_\varepsilon=(\partial_2 V)\chi_\varepsilon$. Because of \eqref{apr14},   $(\partial_1V)\chi_\varepsilon, (\partial_2 V)\chi_\varepsilon \in S^{-m-1}_0$. On the other side, the  partial derivative $\partial_1\chi_\varepsilon$ has compact support which implies that  $V(\partial_1\chi_\varepsilon)\in S^{-m}_{-p}$ for any $p>0$, in particular for $p=m$. Now we  use the same estimates given in the proof of  \cite[Lemma 5.1]{shi2} to conclude that $\widetilde{V}_{\varepsilon}$ is a principal symbol for $\mathcal{V}_{\varepsilon,j}$, and that the remainder terms, coming from  the Taylor expansion, satisfy the required conditions.
\end{proof}
For a measurable function $a:\R^2\to\R_+$ define
$$N(\lambda,a):=\frac{1}{2\pi}vol\{(x,\xi)\in\R^2;a(x,\xi)>\lambda\}.$$
Lemma \ref{shirai} together with  \cite[Lemma 4.7]{dauro} imply that there exists   a positive  $\lambda_0$ such that
$$ n_+(\lambda;Op^W(R_1))= O(N(\lambda,\langle x,\xi\rangle^{-m-1}))=O(\lambda^{-\frac{2}{m+1}}),$$
$$ n_+(\lambda; Op^W(R_2))=O(N(\lambda,\langle x,\xi\rangle^{-m}\langle x\rangle^{-m}))=O(\lambda^{-\frac{1}{2m}-\frac{1}{m}}),$$
{for} $\lambda \in [0,\lambda_0]$. Then,   \eqref{sep25} and the Weyl inequalities imply that for all  $\delta\in(0,1)$
\bel{may2} n_+(\lambda; {\mathcal V}_{\varepsilon, j})\leq n_+((1-\delta)\lambda;Op^W(\widetilde{V}_{\varepsilon}))+ o(\lambda^{-2/m}),\quad \lambda \downarrow 0.\ee
Putting together  \eqref{2},  \eqref{may3}, \eqref{may6}, \eqref{may7}, \eqref{nov13} and \eqref{may2} we obtain that for all $\delta \in (0,1)$
\bel{oct21}{\mathcal N}_j(\lambda)\leq n_+((1-\delta)\lambda;Op^W(\widetilde{V}_{\varepsilon}))+ o(\lambda^{-2/m}),\quad \lambda \downarrow 0.\ee

\begin{lemma}\label{lhomo}
For any $\varepsilon>0$ the function $N(\lambda,\widetilde{V}_\varepsilon)$ satisfies the homogeneity condition \eqref{jul10}
\end{lemma}
\begin{proof}
Note that
$$\begin{array}{lll}&2\pi|N((1-\epsilon)\lambda,\widetilde{V}_\varepsilon)-N((1+\epsilon)\lambda,\widetilde{V}_\varepsilon)|&\\
= &{\displaystyle vol{\{(x,\xi)\in\R^2;(1+\epsilon)\lambda \geq \chi_\varepsilon(b^{-1}(x))V(b^{-1}(x),-\xi)>(1-\epsilon)\lambda\}}}&\\
=&{\displaystyle vol{\{(x,\xi)\in\R^2;(1+\epsilon)\lambda \geq V(b^{-1}(x),-\xi)>(1-\epsilon)\lambda, \,b^{-1}(x)\geq-\varepsilon\}}}\\
&+{vol{\{(x,\xi)\in\R^2;(1+\epsilon)\lambda \geq\chi_\varepsilon(b^{-1}(x))V(b^{-1}(x),-\xi)>(1-\epsilon)\lambda, -2\varepsilon <b^{-1}(x)<-\varepsilon\}}}&\\
\leq&{\displaystyle\int_{\{(x',\xi')\in\R^2;(1+\epsilon)\lambda \geq V(x',\xi')>(1-\epsilon)\lambda, \,x'\geq-\varepsilon\}}B(x')\,\,dx' d\xi'}&\\
&+\,\displaystyle{vol\{(x,\xi)\in\R^2;\,C_{0,0}\,\langle b(x)^{-1},\xi\rangle^{-m}>(1-\epsilon)\lambda, -2\varepsilon<b^{-1}(x)< -\varepsilon\}}\\
\leq&{\displaystyle B_+\left( N((1-\epsilon)\lambda,V,-\varepsilon)-N((1+\epsilon)\lambda,V,-\varepsilon)\right)+O(\lambda^{-1/m})},&\end{array}$$
where in the first inequality we have used the change of variables $b^{-1}(x)=x'$, $-\xi=\xi'$,  that $V$  satisfies \eqref{apr14}  and that  $0\leq\chi_\varepsilon\leq 1$.
Since $N(\lambda,V,-\varepsilon)$ fulfils  \eqref{jul10} we obtain the required result.
\end{proof}

\begin{lemma}\label{lpm}
For any $\varepsilon >0$, $N(\lambda,\widetilde{V}_\varepsilon)$ satisfies condition \eqref{jul10a}. Moreover
\bel{nov24}
 \lim_{\lambda \downarrow 0}\frac{B_+ N(\lambda,V, 0)}{N(\lambda,\widetilde{V}_{\varepsilon})} =1.
 \ee
\end{lemma}

\begin{proof} First let us show that
\bel{jul2}
 \lim_{\lambda \downarrow 0}\frac{N(\lambda,V_\varepsilon)}{N(\lambda,V, 0)}=1.\ee
To see this we  estimate   $|N(\lambda,V_\varepsilon)-N(\lambda,V,0)|$, noticing  that
$\{(x,\xi)\in\R^2;V_\varepsilon(x,\xi)>\lambda\}$ and  $\{(x,\xi)\in\R^2;V(x,\xi)>\lambda, \,x>0\}$ differ in a set contained in
$$  \{(x,\xi)\in\R^2;V_\varepsilon(x,\xi)>\lambda,  -2\varepsilon<x\leq0\}.$$
Then in view of $V_\varepsilon \in S_0^{-m}$ 
$$|N(\lambda,V_\varepsilon)-N(\lambda,V,0)|=  O(\lambda^{-1/m}).$$
Using that $N(\lambda,V,0)$ satisfies property \eqref{jul10a}  we obtain \eqref{jul2}.

Now let us prove that
\bel{jul5}
\lim_{\lambda \downarrow 0}\frac{N(\lambda,\widetilde{V}_{\varepsilon})}{B_+N(\lambda, V_\varepsilon)}=1.
\ee
Similarly to the proof of Lemma \ref{lhomo} we have
\bel{nov29}\begin{array}{ll}
&\displaystyle{2\pi|B_+N(\lambda,{V}_{\varepsilon})-N(\lambda,\widetilde{V}_{\varepsilon})|=\int_{\{(x,\xi);V_\varepsilon(x,\xi)>\lambda\}}B_+-B(x)\,dx d\xi}\\
\leq &\displaystyle{ C\lambda^{-1/m}\int_{-2\varepsilon}^{\lambda^{-1/m}}\langle x \rangle^{-M}\,dx=o(\lambda^{-2/m}),\quad \lambda \downarrow 0.}\\
\end{array}
\ee
Where we have used  \eqref{may14} and \eqref{apr14}, and $C$ is a positive constant independent of $\lambda$. Taking into account  \eqref{jul2} along with \eqref{jul10a}, we  obtain \eqref{jul5}.

Putting together \eqref{jul2} and \eqref{jul5} we get \eqref{nov24}.
\end{proof}

Since $N(\lambda,\widetilde{V}_\varepsilon)$ satisfies \eqref{jul10a} and \eqref{jul10}, it follows that it also  satisfies condition $(\rm{T'})$ of \cite{dauro}. Then  \cite[Theorem 1.3]{dauro} says that
$$
\limsup_{\lambda \downarrow 0}\frac{n_+((1-\delta)\lambda;Op^W(\widetilde{V}_{\varepsilon}))}{N((1-\delta)\lambda,\widetilde{V}_{\varepsilon})}=1,
$$
and therefore \eqref{oct21}, \eqref{nov24}  imply that for all $\delta \in (0,1)$
\bel{aug1}\limsup_{\lambda \downarrow 0}\frac{{\mathcal N}_j(\lambda)}{B_+N((1-\delta)\lambda; {V},0)}\leq 1.\ee

To finish the proof of the upper bound in \eqref{the4} it only remains to note that conditions \eqref{jul10a}, \eqref{jul10} imply that
$$\lim_{\delta \downarrow 0} \limsup_{\lambda \downarrow 0}\frac{N((1-\delta)\lambda,V,0)}{N(\lambda,V,0)}=1.$$

\subsection{Proof of Corollary \ref{V_power}: Lower bound}

Condition \eqref{may14} implies that there exits  a  smooth function $\tilde{B}$ such that $ B(x)\geq\tilde{B}(x)\geq B_-$ for all $x\in \R$, and $B_+-\tilde{B}(x)=\tilde{C}\langle x\rangle ^{-M}$, for some positive constant $\tilde{C}$ and  all $x$ sufficiently big. Using $\tilde{B}$ to define $\tilde{b}$ according to \eqref{b_poten}, we see that  \eqref{vital} implies
\bel{sep25d}({\mathcal E}^+_j-E_j(k)+\lambda)^{-1} \geq \left({\mathcal E}^+_j-E_j(\tilde{b}(b^{-1}(k)),\tilde{b})+\lambda\right)^{-1}, \quad \mbox{for all}\, k \in \R,
\ee
where $E_j(k,\tilde{b})$ is defined as in Lemma \ref{l10}.

Since  $B_+-\tilde{B}$ is strictly decreasing for $x$ large, the function ${\mathcal E}^+_j-E_j((\tilde{b}(b^{-1}(k),\tilde{b})$ is strictly decreasing for $k$ big \cite[Theorem 3.2]{manpu}.  We denote by $\rho_j$ its inverse, which is defined at least in an interval of the form $(0,\gamma)$, $\gamma >0$.  It is obvious that $\lim_{w \downarrow 0} \rho_j(w)=\infty$. Moreover, from Lemma 4.8 in \cite{shi2}, we know that \eqref{may14} implies that  ${\mathcal E}^+_j -E_j(\tilde{b}(b^{-1}(k),\tilde{b})=O(k^{-M})$, $k\to \infty$, and then
\bel{sep29b} \rho_j(w)= O(w^{-1/M}),\quad w\downarrow 0.
\ee
For $j\in\N$,  $\delta \in (0,1)$ and $\lambda >0$ set $\varrho=\varrho(\lambda):=\rho_j(\delta\lambda)$. Then \eqref{sep25d} implies that
\bel{sep29}({\mathcal E}^+_j-E_j(k)+\lambda)^{-1}\geq ((1+\delta)\lambda)^{-1}, \quad \mbox{for all}\,k \geq\varrho(\lambda).\ee
Therefore, for all $r>0$ and $\delta \in (0,1)$
\bel{apr19}\begin{array}{ll}
&n_+\left(r;V^{1/2}T_{j,\infty}(\lambda)V^{1/2}\right)
\geq n_+\left(r;V_\varepsilon^{1/2}T_{j,\infty}(\lambda)V_\varepsilon^{1/2}\right)\\
\geq &n_+\left(r;V_\varepsilon^{1/2}T_{j,\infty}(\lambda,\varrho)V_\varepsilon^{1/2}\right)
\geq n_+\left(r(1+\delta)\lambda;V_\varepsilon^{1/2}P_{j,\infty}(\varrho)V_\varepsilon^{1/2}\right).\\
\end{array}\ee
 In the first and the second inequality we have used the min-max principle, while for the third inequality we used \eqref{sep29}.

Next, using the Weyl inequalities, for any  $\lambda>0$ and $\delta \in (0,1)$
\bel{apr20_a}\begin{array}{ll} &\displaystyle{n_+\left(\lambda ;V_\varepsilon^{1/2}P_{j,\infty}(\varrho)V_\varepsilon^{1/2}\right)}\\
\geq &\displaystyle{ n_+(\lambda(1+\delta) ;V_\varepsilon^{1/2}P_{j,\infty} V_\varepsilon^{1/2})
  - n_+\left(\lambda\delta ;V_\varepsilon^{1/2}(P_{j,\infty}-P_{j,\infty}(\varrho))V_\varepsilon^{1/2}\right).}\end{array}\ee

The term $n_+(\lambda;V_\varepsilon^{1/2}P_{j,\infty} V_\varepsilon^{1/2})=n_+(\lambda;P_{j,\infty} V_\varepsilon P_{j,\infty})=n_+(r;{\mathcal V}_{\varepsilon,j})$ was already obtained in \eqref{may7}, and its asymptotic behavior  can be  estimated as in subsection \ref{sub_upp}.

For  the second term in \eqref{apr20_a} we have that \eqref{nov13} implies
\bel{sep29d}\begin{array}{ll} &\displaystyle{n_+\left(\lambda ;V_\varepsilon^{1/2}(P_{j,\infty}-P_{j,\infty}(\varrho))V_\varepsilon^{1/2}\right)}\\
=&\displaystyle{ n_+\left(\lambda ;\int_{(-\infty,\varrho]}^\oplus \pi_{j,\infty}(k)\,dk \mathcal{F}V_\varepsilon\mathcal{F}^*\int_{(-\infty,\varrho]}^\oplus \pi_{j,\infty}(k)\,dk\right)}\\
=&\displaystyle{n_+\left(\lambda;\one_{(-\infty,\varrho]}{\mathcal F}P_{j,\infty} V_\varepsilon P_{j,\infty}{\mathcal F}^*\one_{(-\infty,\varrho]}\right)=n_+\left(\lambda;\one_{(-\infty,\varrho]}{\mathcal V}_{\varepsilon, j}\one_{(-\infty,\varrho]}\right).}\end{array}\ee

Let $\chi_\lambda(x):=\chi_\varepsilon(-x+\rho_j(\delta\lambda))=\chi_\varepsilon(-x+\varrho(\lambda)$, the same $\chi_\varepsilon$ of the preceding subsection. Then $\chi_\lambda$ is a smooth function with bounded derivatives such that $0\leq \chi_\lambda \leq 1$, $\chi_\lambda(x)=0$ for $x\geq\varrho(\lambda)+2\varepsilon$ and $\chi_\lambda(x)=1$ for $x\leq\varrho(\lambda)+\varepsilon$. It is important to note that for all positive $\lambda$, $\varepsilon$, and  $\delta \in (0,1)$,  $\chi_\lambda \in S_0^0$ and  its semi-norms $n_{\alpha,\beta}^{0,0}(\chi_\lambda)$ (defined by \eqref{sep25b}) are independent of $\delta$ and $\lambda$ for all $(\alpha,\beta)\in\Z_+^2$. Indeed,
$$n_{\alpha,\beta}^{0,0}(\chi_\lambda)=\left\{\begin{array}{ll} 0&;\alpha>0\\||\chi_\varepsilon^{(\beta)}||_{L^\infty(\R)}&;\alpha=0.\end{array}\right.$$

Write as before  $Op^W(\widetilde{V}_\varepsilon)$ for the  Pseudo-differential operator with Weyl symbol $\widetilde{V}_\varepsilon$. Then, since for all $\lambda >0,$ $\chi_\lambda \one_{(-\infty,\varrho(\lambda)]}=\one_{(-\infty,\varrho(\lambda)]}$, we have   that
 \bel{may11}\begin{array}{ll}&\one_{(-\infty,\varrho]}{\mathcal V}_{\varepsilon,j}\one_{(-\infty,\varrho]}=\one_{(-\infty,\varrho]}Op^W(\widetilde{V}_\varepsilon)\one_{(-\infty,\varrho]}+
 \one_{(-\infty,\varrho]}\left({\mathcal V}_{\varepsilon,j}-Op^W(\widetilde{V}_\varepsilon)\right)\one_{(-\infty,\varrho]}\\
  =&\one_{(-\infty,\varrho]}\left(Op^W(\widetilde{V}_\varepsilon)Op^W(\chi_\lambda)\right)\one_{(-\infty,\varrho]}+\one_{(-\infty,\varrho]}\left({\mathcal V}_{\varepsilon,j}-Op^W(\widetilde{V}_{\varepsilon})\right)\one_{(-\infty,\varrho]}.\end{array}
\ee
The symbol $\widetilde{V}_\varepsilon \in S^{-m}_0$ and $\chi_\lambda \in S_0^0$, then it is well known that \cite[Theorem 18.5.4]{hor3}  \bel{oct8}Op^W(\widetilde{V}_\varepsilon)Op^W(\chi_\lambda)=Op^W(\widetilde{V}_\varepsilon\chi_\lambda)+Op^W(R_\lambda),\ee
where $R_\lambda\in S_0^{-m-1}$, and each one of  its semi-norms $n_{\alpha,\beta}^{0, -m-1}(R_\lambda)$ is polynomially bounded by a finite number of  semi-norms of $\widetilde{V}_\varepsilon$ and $\chi_\lambda$ in $S_0^{-m}$ and  $S_0^0$, respectively. Since the semi-norms of $\chi_\lambda$ are independent of $\lambda$,   \cite[Lemma 4.7]{dauro} implies that there exists a positive constants  $\lambda_0$ such that
\bel{oct8a} n_+(\lambda;Op^W(R_\lambda))=O(\lambda^{-2/(m+1)}), \quad \mbox{for}\,\lambda \in (0,\lambda_0].\ee

\begin{lemma}\label{sep}
For every $\varepsilon >0$
$$\lim_{\lambda \downarrow 0} \frac{n_+(\lambda;Op^W(\widetilde{V}_\varepsilon\chi_\lambda))}{\lambda^{-2/m}}=0.$$
\end{lemma}
\begin{proof}
By  \cite[Proposition 4.1]{dauro}, there are positive constants $C_\lambda$ and  $\zeta$ such that
\bel{oct6}n_+(\lambda;Op^W(\widetilde{V}_\varepsilon\chi_\lambda))\leq N(\lambda,\widetilde{V}_\varepsilon\chi_\lambda)+C_\lambda \lambda^{(-2/m)+\zeta}.\ee
The constant $C_\lambda$ depends polynomially on a finite number of semi-norms of the  symbol $\widetilde{V}_\varepsilon\chi_\lambda$, but  from composition of symbols each one of  the semi-norms of $\widetilde{V}_\varepsilon\chi_\lambda$ is polynomially bounded by a finite number of semi-norms of  $V$ in $S_0^{-m}$ and $\chi_\lambda$ in $S_0^0$. Consequently, the constant $C_\lambda$ can be taken independent of $\lambda$.

The proof of \eqref{oct6} that appears in \cite{dauro} is for symbols that do not depend on $\lambda$. However, it works as well in our case  just introducing  minor changes.

Now, since  $(\widetilde{V}_\varepsilon\chi_\lambda)(x,\xi)=V(b^{-1}(x),-\xi)\chi_\varepsilon(b^{-1}(x))\chi_\lambda(x)$, where the support of $\chi_\varepsilon(b^{-1}(x))\chi_\lambda(x)$ is contained on the strip $\{(x,\xi)\in\R^2;b(-2\varepsilon)\leq x\leq \varrho(\lambda)+2\varepsilon\}$, and  $V$ is in $S_0^{-m}$,   the set $\{(x,\xi)\in \R^2;\widetilde{V}_\varepsilon\chi_\lambda(x,\xi)>\lambda\}$ is contained in
$$ \{(x,\xi)\in \R^2;\langle b^{-1}(x),\xi\rangle^{-m}>\lambda,\, b(-2\varepsilon)\leq x\leq \varrho(\lambda)+2\varepsilon\}.$$
Then,
\bel{aug19}N(\lambda,\widetilde{V}_\varepsilon\chi_\lambda)=O(\lambda^{-1/m}\,((\varrho(\lambda)+2\varepsilon)-b(-2\varepsilon))).\ee
Putting together \eqref{oct6}, \eqref{aug19} and \eqref{sep29b}  (recalling that $M>m$), we finish the proof of the Lemma.
\end{proof}

Gathering   \eqref{may11}, Lemma \ref{shirai}, \eqref{oct8}, \eqref{oct8a} and  Lemma \ref{sep} we obtain
\bel{sep29c}\lim_{\lambda\downarrow 0}\frac{n_+\left(\lambda;\one_{(-\infty,\varrho]}{\mathcal V}_{\varepsilon_j}\one_{(-\infty,\varrho]}\right)}{\lambda^{-2/m}}=0,
\ee
thus, for all $\delta \in (0,1)$ \eqref{2}, \eqref{apr19}, \eqref{apr20_a}, \eqref{sep29d}, \eqref{sep29c} and Lemma \ref{lpm} imply
$$\liminf_{\lambda \downarrow 0} \frac{{\mathcal N}_j(\lambda)}{B_+N(\lambda(1+\delta), {V},0)}\geq\liminf_{\lambda \downarrow 0}\frac{n_+(\lambda(1+\delta);Op^W(\tilde{V}_\varepsilon))}{N(\lambda(1+\delta), \widetilde{V}_\varepsilon)}.$$
Finally, arguing as in the last part of  subsection \ref{sub_upp} we can obtain the lower bound in \eqref{the4}.\\

{\bf Acknowledgements.} This Project was partially supported by  Vicerrector\'ia de la Investigaci\'on de la Pontificia Universidad Cat\'olica de Chile, grant \emph{Inicio} 43/2014.\\
\bibliographystyle{plain}
\bibliography{biblio}

\begin{thebibliography}{10}

\bibitem{bershu}
F.~Berezin and M.~Shubin.
\newblock {\em The {S}chr\"odinger equation}.
\newblock Kluwer Academic Publishers Group, Dordrecht, 1991.

\bibitem{birsol}
M.~Sh. Birman and M.~Z. Solomjak.
\newblock {\em Spectral theory of selfadjoint operators in Hilbert space}.
\newblock Mathematics and its Applications (Soviet Series). D. Reidel
  Publishing Co., Dordrecht, 1987.

\bibitem{bmr1}
V.~Bruneau, P.~Miranda, and G.~Raikov.
\newblock Discrete spectrum of quantum {H}all effect {H}amiltonians {I}.
  {M}onotone edge potentials.
\newblock {\em J. Spectr. Theory}, 1:237--272, 2011.

\bibitem{bmr2}
V.~Bruneau, P.~Miranda, and G.~Raikov.
\newblock {D}irichlet and {N}eumann eigenvalues for half-plane magnetic
  {H}amiltonians.
\newblock {\em Rev. Math. Phy.}, 26, 2014.

\bibitem{cal}
M.~Calvo.
\newblock Exactly soluble two-dimensional electron gas in a magnetic-field
  barrier.
\newblock {\em Phys. Rev. B}, 48:2365--2369, 1993.

\bibitem{dauro}
M.~Dauge and D.~Robert.
\newblock Weyl's formula for a class of pseudodifferential operators with
  negative order on ${L}^2({R}^n)$.
\newblock In {\em {P}seudodifferential operators ({O}berwolfach, 1986)}, volume
  1256 of {\em Lecture Notes in Math.}, pages 91--112. Springer, 1987.

\bibitem{dpccgf}
N.~Davies, A.~Patel, A.~Cortijo, V.~Cheianov, F.~Guinea, and V.~Fal'ko.
\newblock Skipping and snake orbits of electrons: Singularities and
  catastrophes.
\newblock {\em Phys. Rev. B}, 85:155433--155437, 2012.

\bibitem{dgr}
N.~Dombrowski, F.~Germinet, and G.~Raikov.
\newblock Quantization of edge currents along magnetic barriers and magnetic
  guides.
\newblock {\em Ann. Henri Poincar\'e}, 12:1169--1197, 2011.

\bibitem{dhs}
N.~Dombrowski, P.~Hislop, and E.~Soccorsi.
\newblock Edge currents and eigenvalue estimates for magnetic barrier
  {S}chr\"odinger operators.
\newblock {\em Asymptot. Anal.}, 89:331--363, 2014.

\bibitem{evgi}
W.~N. Everitt and M.~Giertz.
\newblock Some inequalities associated with certain ordinary differential
  operators.
\newblock {\em Math. Z.}, 126:308--326, 1972.

\bibitem{exkov}
P.~Exner and H.~Kova{\u r}\'{\i}k.
\newblock Magnetic strip waveguides.
\newblock {\em J. Phys. A}, 33:3297--3311, 2000.

\bibitem{filpush}
N.~Filonov and A.~Pushnitski.
\newblock Spectral asymptotics of {P}auli operators and orthogonal polynomials
  in complex domains.
\newblock {\em Comm. Math. Phys.}, 264:759--772, 2006.

\bibitem{fh}
S.~Fournais and B.~Helffer.
\newblock {\em Spectral Methods in Surface Superconductivity}.
\newblock Progress in Nonlinear Differential Equations and their Applications.
  Birkh\"auser Boston, Inc., Boston, MA, 2010.

\bibitem{gmz}
R.~Gerhardts, A.~Manolescu, and S.~Zwerschke.
\newblock Planar cyclotron motion in unidirectional superlattices defined by
  strong magnetic and electric fields: {T}races of classical orbits in the
  energy spectrum.
\newblock {\em Phys. Rev. B}, 60:5536--5548, 1999.

\bibitem{hprs}
P.~Hislop, N.~Popoff, N.~Raymond, and M.~Sunqvist.
\newblock Band functions in the presence of magnetic steps.
\newblock {\em arXiv:1501.02824}.

\bibitem{hissocc2}
P.~D. Hislop and E.~Soccorsi.
\newblock Edge states induced by {I}watsuka {H}amiltonians with positive
  magnetic fields.
\newblock {\em J. Math. Anal. Appl.}, 422:594--624, 2015.

\bibitem{hor3}
L.~H\"ormander.
\newblock {\em The Analysis of {L}inear {P}artial {D}ifferential {O}perators
  III. Pseudodifferential {O}perators.}
\newblock Grundlehren der Mathematischen Wissenschaften, 274. Springer-Verlag,
  Berlin, 1985.

\bibitem{ivrii}
V.~Ivrii.
\newblock {\em Microlocal Analysis and Precise Spectral Asymptotics.}
\newblock Springer Monographs in Mathematics. Springer-Verlag, Berlin, 1998.

\bibitem{iwa}
A.~Iwatsuka.
\newblock Examples of absolutely continuous {S}chr\"odinger operators in
  magnetic fields.
\newblock {\em Publ. Res. Inst. Math. Sci.}, 21:385--401, 1985.

\bibitem{lww}
H.~Leschke, S.~Warzel, and A.~Weichlein.
\newblock Energetic and dynamic properties of a quantum particle in a spatially
  random magnetic field with constant correlations along one direction.
\newblock {\em Ann. Henri Poincar\'e}, 7:335--363, 2006.

\bibitem{manpu}
M.~M{\u a}ntoiu and R.~Purice.
\newblock Some propagation properties of the {I}watsuka model.
\newblock {\em Comm. Math. Phys.}, 188:691--708, 1997.

\bibitem{meelroz}
M.~Melgaard and G.~Rozenblum.
\newblock Eigenvalue asymptotics for weakly perturbed {D}irac and
  {S}chr{\"o}dinger operators with constant magnetic fields of full rank.
\newblock {\em Comm. Partial Differential Equations}, 28:697--736, 2003.

\bibitem{mu}
J.~Muller.
\newblock Effect of a nonuniform magnetic field on a two-dimensional electron
  gas in the ballistic regime.
\newblock {\em Phys. Rev. Lett.}, 68:385--388, 1992.

\bibitem{pr}
F.~Peeters and J.~Reijniers.
\newblock Snake orbits and related magnetic edge states.
\newblock {\em J. Phys.: Condens. Matter}, 12:9771--9786, 2000.

\bibitem{per}
M.~Persson.
\newblock Eigenvalue asymptotics of the even-dimensional exterior
  {L}andau-{N}eumann {H}amitonian.
\newblock {\em Adv. Math. Phys.}, 2009.

\bibitem{popoff}
N.~Popoff.
\newblock Sur le spectre de l'op\'erateur de {S}chr\"odinger magn\'etique dans
  un domaine di\'edral.
\newblock {\em PhD Thesis, Universit\'e de Rennes 1}, 2012.

\bibitem{proz}
A.~Pushnitski and G.~Rozenblum.
\newblock Eigenvalue clusters of the {L}andau {H}amiltonian in the exterior of
  a compact domain.
\newblock {\em Doc. Math.}, 12:569--586, 2007.

\bibitem{pushroz2}
A.~Pushnitski and G.~Rozenblum.
\newblock On the spectrum of {B}argmann-{T}oeplitz operators with symbols of a
  variable sign.
\newblock {\em J. Anal. Math.}, 114:317--340, 2011.

\bibitem{rai1}
G.~Raikov.
\newblock Eigenvalue asymptotics for the {S}chr\"odinger operator with
  homogeneous magnetic potential and decreasing electric potential. {I}.
  {B}ehaviour near the essential spectrum tips.
\newblock {\em Commun. Partial Differential Equations}, 15:407--434, 1990.

\bibitem{raiwar}
G.~Raikov and S.~Warzel.
\newblock Quasi-classical versus non-classical spectral asymptotics for
  magnetic {S}chr\"odinger operators with decreasing electric potentials.
\newblock {\em Rev. Math. Phys.}, 14:1051--1072, 2002.

\bibitem{rmcpv}
J.~Reijniers, A.~Matulis, K.~Chang, F.~M. Peeters, and P.~Vasilopoulos.
\newblock Confined magnetic guiding orbit states.
\newblock {\em Europhys. Lett.}, 59:749, 2002.

\bibitem{roz}
G.~Rozenblum.
\newblock On lower eigenvalue bounds for {T}oeplitz operators with radial
  symbols in {B}ergman spaces.
\newblock {\em J. Spectr. Theory}, 1:299--325, 2011.

\bibitem{rozta1}
G.~Rozenblum and G.~Tashchiyan.
\newblock On the spectral properties of the perturbed {L}andau {H}amiltonian.
\newblock {\em Comm. Partial Differential Equations}, 33:1048--1081, 2008.

\bibitem{rozta2}
G.~Rozenblum and G.~Tashchiyan.
\newblock On the spectral properties of the {L}andau {H}amiltonian perturbed by
  a moderately decaying magnetic field.
\newblock In {\em Spectral and {S}cattering {T}heory for {Q}uantum {M}agnetic
  {S}ystems}, volume 500 of {\em Contemp. Math.}, pages 169--186. Amer. Math.
  Soc., Providence, RI, 2009.

\bibitem{shi1}
S.~Shirai.
\newblock Eigenvalue asymptotics for the {S}chr\"odinger operator with steplike
  magnetic field and slowly decreasing electric potential.
\newblock {\em Publ. Res. Inst. Math. Sci.}, 39:297--330, 2003.

\bibitem{shi2}
S.~Shirai.
\newblock Strong-electric-field eigenvalue asymptotics for the {I}watsuka
  model.
\newblock {\em J. Math. Phys.}, 46, 2005.

\end{thebibliography}
{\sc Pablo Miranda}\\
Departamento de Matem\'atica,
Facultad de Matem\'aticas,\\
Pontificia Universidad Cat\'olica de Chile,
Vicu\~na Mackenna 4860, Santiago de Chile\\
E-Mail: pmirandar@mat.puc.cl\\
\end{document}